\documentclass[a4,10pt]{article}

\usepackage{wrapfig}

\usepackage{lineno}

\usepackage{amsmath}
\usepackage{amsthm}
\usepackage{amssymb}
\usepackage{enumerate}
\usepackage[pdftex]{graphicx}
\usepackage{float}

\usepackage{url}

\usepackage[top=36truemm,bottom=35truemm,left=32truemm,right=32truemm]{geometry}

\usepackage{titlesec}
\titleformat*{\section}{\large\bfseries}

\theoremstyle{definition}
\newtheorem{theorem}{Theorem}[section]

\newtheorem{proposition}[theorem]{Proposition}
\newtheorem{remark}[theorem]{Remark}
\numberwithin{equation}{section}

\allowdisplaybreaks

\begin{document}
  \title{\Large{Generalized Hardy's identity \\for the astroid-type $p$-circle lattice point problem}}
  \author{Masaya Kitajima}
  \date{}
  \maketitle
  \begin{abstract}
  Let $r$ be a positive real number and $p$ satisfy $(2/p)\in\mathbb{N}$. Then, we consider the lattice point problem of the closed curves \textit{astroid-type $p$-circle} $\{x\in\mathbb{R}^{2}|\ |x_{1}|^{p}+|x_{2}|^{p}=r^{p}\}$ which generalize the circle. In investigating the asymptotic behavior of the error term in the area approximation of the circle, G.H. Hardy conjectured an infimum for the evaluation in 1917. One of the grounds for this conjecture is \textit{the Hardy's identity}, which is a series representation of the term, consisting of the Bessel function of order one and a certain number-theoretic function. \par
  In order to investigate an infimum in the error evaluation of the astroid-type $p$-circle, which is unknown in previous studies, in this paper, we derive generalized Hardy's identity for the figures by using generalized Bessel functions. Furthermore, the differential formula for the functions, which is important for the proof of this identity, is closely related to \textit{the Erd\'{e}lyi-Kober operator}, and this formula and operator are expected to be useful in our future research. 
\\
  \end{abstract}
  %% ========================================================================
  
  \section{Introduction and main results}
  \hspace{13pt}For positive real numbers $p$ and $r$, we call the closed curve defined by $\{x\in\mathbb{R}^{2}|\ |x_{1}|^{p}+|x_{2}|^{p}=r^{p}\}$ \textit{$p$-circle} of radius $r$ (\textit{Lam\'{e}'s curve}, or \textit{supper ellipse}), which is a generalization of the circle (see Figure 1). Then, focusing on the lattice points inside the $p$-circle, we consider the error between the area of the $p$-circle itself and the area of the approximated mosaic (the number of lattice points inside the $p$-circle) 
  \begin{equation}\label{P_p}
  P_{p}(r):=N_{p}(r)-\frac{2}{p}\frac{\Gamma(\frac{1}{p})^{2}}{\Gamma(\frac{2}{p})}r^{2}
  \end{equation}
  by approximating the closed curve as a mosaic, as shown in the right side of Figure 1. Then, we pose the lattice point problem of $p$-circle as the problem of finding a value $\alpha_{p}$ such that $P_{p}(r)=\mathcal{O}(r^{\alpha_{p}})$ and $P_{p}(r)=\Omega(r^{\alpha_{p}})$ are satisfied, which has been studied for a long time. Here, $N_{p}$ is the number of lattice points in the $p$-circle, $\Gamma$ is the gamma function, and, for the functions $f$ and $g$, $f(t)=\mathcal{O}(g(t))$ (resp. $f(t)=\Omega(g(t))$) mean $\limsup_{t\to\infty}|\frac{f(t)}{g(t)}|<+\infty$ (resp. $\liminf_{t\to\infty}|\frac{f(t)}{g(t)}|>0$). \par\vspace{5pt}
  
  In particular, the case $p=2$ (that is, when the closed curve is a circle) is called the Gauss' circle problem\cite{Gauss}, and in 1917, G.H. Hardy\cite{Hardy-1917} conjectured that an infimum of $\mathcal{O}$ evaluation for $P_{2}$ is $\frac{1}{2}$-order (\textit{Hardy's conjecture}), based on his previous results\cite{Hardy-1915} which established $P_{2}(r)\neq\mathcal{O}(r^{\frac{1}{2}})$, $P_{2}(r)=\Omega(r^{\frac{1}{2}}(\log r)^{1/4})$ and \textit{Hardy's identity} (for example, see also \cite{Kratzel}, Theorem 3.12)
  \begin{equation}\label{HI}
  \hspace{-160pt}P_{2}(r)=r\sum_{k=1}^{\infty}\frac{R(k)}{k^{\frac{1}{2}}}J_{1}(2\pi k^{\frac{1}{2}}r),
 \end{equation}\vspace{-10pt}
 \begin{equation*}
  \hspace{120pt}\text{with}\text{ $J_{\omega}$: the Bessel function of order }\omega,\text{ $R(k):=\#\{n\in\mathbb{Z}^{2}|\ |n|^{2}=k\}$}.
\end{equation*}\par
  Since then, many mathematicians have improved the evaluation (the latest result was obtained by M.N. Huxley \cite{Huxley-2003} in 2003 ($P_{2}(r)=\mathcal{O}(r^{\frac{1}{2}+\varepsilon}$) for $\varepsilon>\frac{27}{208}$), but this problem has not been completely solved yet.
  
  \begin{figure}[t]
    \centering
    \includegraphics[width=0.8\linewidth]{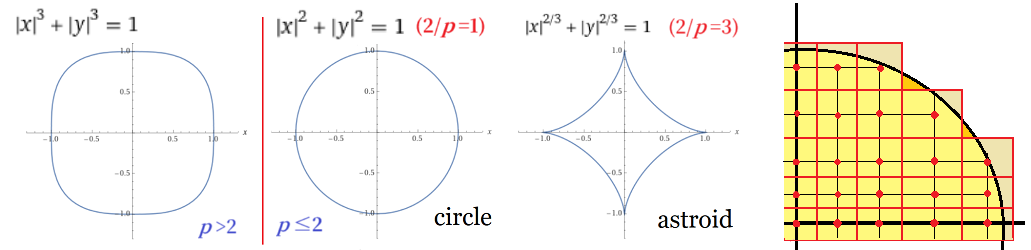}
    \label{fig}
    \caption{Examples of the $p$-circle and the approximation by unit squares.}
  \end{figure}
  \par\vspace{5pt}
  
  On the other hand, for the cases $p>2$, the following important theorem by E. Kr\"{a}tzel is given by the decomposition (\cite{Kratzel}, (3.57))
  \begin{equation*}
    P_{p}(r)=\Psi(r;p)+\Delta(r;p)
  \end{equation*}
  with the second main term $\Psi(r;p)$ (\cite{Kratzel}, (3.55)), which is represented as the series expansion consisting of the generalized Bessel functions (\cite{Kratzel}, Definition 3.3). It can be said that the application of the result in 1993 by G. Kuba\cite{Kuba} (the $\mathcal{O}$ estimate of the remainder term $\Delta(r;p)$) have solved the problem for the cases $p>\frac{73}{27}$.
  \begin{theorem}[\textit{\cite{Kratzel}, Theorem 3.17 A}]\label{upper}
    \itshape{Let $p>2$. If $\beta_{p}<1-\frac{1}{p}$ such that $\Delta(r;p)=\mathcal{O}(r^{\beta_{p}})$ exists, then $P_{p}(r)=\mathcal{O}(r^{1-\frac{1}{p}}),\Omega(r ^{1-\frac{1}{p}})$ holds.}
  \end{theorem}
  \par
  Now, among the remaining cases $0<p<2$, we specifically study the cases $p$ satisfying $\frac{2}{p}\in\mathbb{N}$ (we call the corresponding closed curve \textit{astroid-type $p$-circle}. See $\cite{K1},\cite{K2}$). For these cases, no previous study other than our trials has been done. Even if we apply the Kr\"{a}tzel's method established for the above-mentioned cases $p>2$ to these cases, singularities arise in the functions we mainly used, which makes the approach to the problem difficult. \par
  However, we would like to emphasize that another problem for lattice points has been widely studied, including the cases $0<p<2$, and that these cases are also considered interesting research subjects. For example, for a family of the $p$-ellipse $\{x\in\mathbb{R}^{2}|\ |sx_{1}|^{p}+|\frac{x_{2}}{s}|^{p}=r^{p}\}\ (s>0)$, the general form of the $p$-circle, there are problems to find $s$ such that the number of lattice points is the largest. For $p\neq1$, this problem has been solved by R.S. Laugesen et al. (see \cite{Laugesen1}\cite{Laugesen2}). \vspace{5pt}\par
  With this background, we approach the cases $p$ satisfying $\frac{2}{p}\in\mathbb{N}$ by using a new method. Considering the fact that the best estimate conjecture (Hardy's conjecture) of the circle problem is based on Hardy's identity (\ref{HI}), we will try to generalize it as a starting point. We focus on the paper \cite{Kuratsubo-2022} by S. Kuratsubo and E. Nakai in 2022, which gave a harmonic-analytic claim equivalent to Hardy's conjecture by using the Fourier transform of spherically symmetric functions and the Bessel function $J_{\omega}$ as an important key. Inspired by this paper for a circle, we generalized the Bessel functions on $\mathbb{R}^{2}$ (\cite{K1}, Definition 2.5)
  \begin{equation*}\label{I-Om}
      J_{\omega}^{[p]}(x):=
      \begin{cases}
        \frac{1}{\Gamma(\frac{1}{p})^{2}}\left(\frac{2}{p}
    \right)^{2}\int_{0}^{1}\cos(x_{1}t^{\frac{1}{p}})\cos(x_{2}(1-t)^{\frac{1}
    {p}})t^{\frac{1}{p}-1}(1-t)^{\frac{1}{p}-1}dt\qquad\text{if }\omega=0,\\
        \frac{|x|_{p}^{\omega}}{p^{\omega-1}\Gamma(\omega)}\int_{0}^{1}
        J_{0}^{[p]}(\tau x)\tau(1-\tau^{p})^{\omega-1}
        d\tau\qquad\text{if }\omega>0, \qquad\text{for }x\in\mathbb{R}^{2}
      \end{cases}
  \end{equation*}
  based on \textit{$p$-radial} (a generalization of spherical symmetry), in order to tackle the problem for the astroid-type $p$-circle. Note that $|\xi|_{p}(:=(|\xi_{1}|^{p}+|\xi_{2}|^{p})^{\frac{1}{p}})$ is $p$-norm on $\mathbb{R}^{2}$.\par\vspace{7pt}
  Now, as a main result in this paper, we obtained the following generalization of the Hardy's identity (\ref{HI}) for the astroid-type $p$-circle by using a function on $\mathbb{R}_{\geq0}$
    \begin{equation}\label{mathcalJ}
      \mathcal{J}_{\omega,\varphi}^{[p]}(r):=\frac{(\frac{2}{p})^{2}}{\Gamma(\frac{1}{p})^{2}}\sum_{k=0}^{\infty}\frac{p^{2k}(-1)^{k}}{\Gamma(\frac{2}{p}(k+1)+\omega)}\Bigl(\frac{r}{p}\Bigr)^{2k+\omega}\Bigl(\sum_{m\in\mathbb{N}_{0}^{2}\ |m|'=k}\frac{\Gamma(\frac{2m+1}{p})}{(2m)!}|\cos^{m_{1}}\varphi\sin^{m_{2}}\varphi|^{\frac{4}{p}}\Bigr), \vspace{-5pt}
    \end{equation}
  which is defined by fixing \textit{a distorted angle} $\varphi$ via $x=(\mathrm{sgn}(\cos\varphi)r|\cos\varphi|^{\frac{2}{p}}, \mathrm{sgn}(\sin\varphi)r|\sin\varphi|^{\frac{2}{p}})$ ($0\leq\varphi<2\pi,r\geq0$) from the series representation of $J_{\omega}^{[p]}$(\cite{K1}, Proposition 2.6)
  \begin{equation}\label{series-omega}
      J_{\omega}^{[p]}(x)=\frac{(\frac{|x|_{p}}{p})^{\omega}(\frac{2}{p})^{2}}
      {\Gamma(\frac{1}{p})^{2}}\sum_{k=0}^{\infty}\frac{(-1)^{k}}{\Gamma(\frac{2}{p}(k+1)
      +\omega)}\sum_{m\in\mathbb{N}_{0}^{2}\ |m|'=k}\frac{\Gamma(\frac{2m+1}{p})}{(2m)!}x^{2m}\qquad\text{for }\omega\geq0
  \end{equation}
  (note that, for $n\in\mathbb{N}_{0}^{2}$, $x\in\mathbb{R}^{2}$ and $\xi\in\mathbb{R}_{>0}^{2}$, we denote $|n|':=n_{1}+n_{2},\ n!:=n_{1}!\cdot n_{2}!,\ x^{n}:=x_{1}^{n_{1}}\cdot x_{2}^{n_{2}}$ and $\Gamma(\xi):=\Gamma(\xi_{1})\Gamma(\xi_{2})$ according to multi-index notation).
    \begin{theorem}[\textit{Generalized Hardy's identity for the astroid-type $p$-circle}]\label{GHI}\ \\\itshape{Let $p$ satisfy $\frac{2}{p}\in\mathbb{N}$ and a finite set $\mathcal{A}_{s}^{[p]}$ consist of distorted angles $\varphi$ corresponding to lattice points on $p$-circle of radius $s^{1/p}\ (\geq1)$. Specifically, $\mathcal{A}_{s}^{[p]}$ is denoted as follows, and $\#\mathcal{A}_{s}^{[p]}\leq4[s^{\frac{1}{p}}]$ holds. 
    \begin{equation}\label{mathcalA}
    \mathcal{A}_{s}^{[p]}:=\{\varphi\in[0,2\pi)\ |\ (\mathrm{sgn}(\cos\varphi)s^{\frac{1}{p}}|\cos\varphi|^{\frac{2}{p}},\mathrm{sgn}(\sin\varphi)s^{\frac{1}{p}}|\sin\varphi|^{\frac{2}{p}})\in\mathbb{Z}^{2}\}.
    \end{equation}
    Then, the following holds for the counting measure $\mu$.}
      \begin{equation*}\label{HI-p}
    P_{p}(r)=\frac{p\Gamma(\frac{1}{p})^{2}}{2\pi}r\int_{1}^{\infty}\frac{1}{s^{\frac{1}{p}}}\Bigl(\sum\nolimits_{\varphi\in\mathcal{A}_{s}^{[p]}}\mathcal{J}_{1,\varphi}^{[p]}(2\pi s^{\frac{1}{p}}r)\Bigr)d\mu(s). 
      \end{equation*}
    \end{theorem}
    \vspace{20pt}\par
  In Section 2, we clarify the necessary properties of the generalized Bessel functions $J_{\omega}^{[p]}$, $J_{\omega,\varphi}^{[p]}$ and prove the main result, and in the next Section 3, we conclude this paper by describing the relevance of our previous results and the prospects of our research.

  %% ========================================
  \section{Proof of Theorem \ref{GHI} (Generalized Hardy's Identity)}
  
  %%%%%%%%%%%%%%%%%%%%%%%
  \subsection{\normalsize{The properties of the generalized Bessel functions}}
  \hspace{13pt}Firstly, for later use, we rewrite the series representation (\ref{mathcalJ}) of $\mathcal{J}_{\omega,\varphi}^{[p]}$ as follows. 
   \begin{align*}
      \sum_{m\in\mathbb{N}_{0}^{2}\ |m|'=k}&\frac{\Gamma(\frac{2m+1}{p})}{(2m)!}|\cos^{m_{1}}\varphi\sin^{m_{2}}\varphi|^{\frac{4}{p}}\\
      &\hspace{-30pt}=\sum_{n=0}^{k}\frac{\Gamma(\frac{1}{p}(2n+1))\Gamma(\frac{1}{p}(2(k-n)+1))}{(2n)!(2(k-n))!}(\cos^{\frac{4}{p}}\varphi)^{n}(\sin^{\frac{4}{p}}\varphi)^{k-n}\\
      &\hspace{-30pt}=\sum_{n=0}^{k}\frac{\pi\Gamma(\frac{1}{p}(2n+1))\Gamma(\frac{1}{p}(2(k-n)+1))}{n!\ 2^{2n}\Gamma(n+\frac{1}{2})\cdot(k-n)!\ 2^{2(k-n)}\Gamma(k-n+\frac{1}{2})}(\cos^{\frac{4}{p}}\varphi)^{n}(\sin^{\frac{4}{p}}\varphi)^{k-n}\\
      &\hspace{-30pt}=\frac{\pi}{2^{2k}}\sum_{n=0}^{k}\frac{1}{n!\ (k-n)!}\Bigl(\frac{\Gamma(\frac{2}{p}(n+\frac{1}{2}))\Gamma(\frac{2}{p}(k-n+\frac{1}{2}))}{\Gamma(n+\frac{1}{2})\Gamma(k-n+\frac{1}{2})}\Bigr)(\cos^{\frac{4}{p}}\varphi)^{n}(\sin^{\frac{4}{p}}\varphi)^{k-n}\\
      &\hspace{-30pt}=\frac{\pi}{2^{2k}}\sum_{n=0}^{k}\frac{\Gamma(\frac{2}{p}(k+1))}{n!\ (k-n)!}\Bigl(\frac{\Gamma(\frac{2}{p}(n+\frac{1}{2}))\Gamma(\frac{2}{p}(k-n+\frac{1}{2}))}{\Gamma(\frac{2}{p}(k+1))}\Bigr)\Bigl(\frac{\Gamma(k+1)}{\Gamma(n+\frac{1}{2})\Gamma(k-n+\frac{1}{2})}\Bigr)\frac{(\cos^{\frac{4}{p}}\varphi)^{n}(\sin^{\frac{4}{p}}\varphi)^{k-n}}{\Gamma(k+1)}\\
      &\hspace{-30pt}=\frac{\pi}{k!\ 2^{2k}}\sum_{n=0}^{k}\frac{(\frac{2}{p}(k+1)-1)!}{n!\ (k-n)!}\Bigl(\frac{B(\frac{2}{p}(n+\frac{1}{2}),\frac{2}{p}(k-n+\frac{1}{2}))}{B(n+\frac{1}{2},k-n+\frac{1}{2})}\Bigr)(\cos^{\frac{4}{p}}\varphi)^{n}(\sin^{\frac{4}{p}}\varphi)^{k-n}
      \quad=:\frac{\pi}{k!\ 2^{2k}}\Phi_{k,\varphi}^{[p]},
    \end{align*}\vspace{-8pt}
    \begin{equation*}
      \text{that is, }\mathcal{J}_{\omega,\varphi}^{[p]}(r)=\frac{(\frac{2}{p})^{2}}{\Gamma(\frac{1}{p})^{2}}\sum_{k=0}^{\infty}\frac{p^{2k}(-1)^{k}}{\Gamma(\frac{2}{p}(k+1)+\omega)}\Bigl(\frac{r}{p}\Bigr)^{2k+\omega}\frac{\pi}{k!\ 2^{2k}}\Phi_{k,\varphi}^{[p]}.
      \end{equation*}\par
      Note that $B(s,t)$ is the beta function. Thus, we obtain the following series representation.
    \begin{equation}\label{mathcalJ-1}
      \mathcal{J}_{\omega,\varphi}^{[p]}(r)=\Bigl(\frac{2}{p}\Bigr)^{2+\omega}\frac{\pi}{\Gamma(\frac{1}{p})^{2}}\sum_{k=0}^{\infty}\frac{(-1)^{k}}{k!\Gamma(\frac{2}{p}(k+1)+\omega)}\Bigl(\frac{r}{2}\Bigr)^{2k+\omega}\Phi_{k,\varphi}^{[p]},
    \end{equation}\vspace{-5pt}
    \begin{equation}\label{Phi}
      \text{with }\Phi_{k,\varphi}^{[p]}=\sum_{n=0}^{k}\frac{(\frac{2}{p}(k+1)-1)!}{n!\ (k-n)!}\Bigl(\frac{B(\frac{2}{p}(n+\frac{1}{2}),\frac{2}{p}(k-n+\frac{1}{2}))}{B(n+\frac{1}{2},k-n+\frac{1}{2})}\Bigr)(\cos^{\frac{4}{p}}\varphi)^{n}(\sin^{\frac{4}{p}}\varphi)^{k-n}.
    \end{equation}
    \par
    In particular, for $p=2$, these functions are invariant to the distorted angles $\varphi$. 
    \begin{align*}
      \Phi_{k,\varphi}^{[2]}&=\sum_{n=0}^{k}\frac{k!}{n!\ (k-n)!}\cdot1\cdot(\cos^{2}\varphi)^{n}(\sin^{2}\varphi)^{k-n}=(\cos^{2}\varphi+\sin^{2}\varphi)^{k}=1,\\
      \mathcal{J}_{\omega,\varphi}^{[2]}(r)&=\frac{\pi}{\Gamma(\frac{1}{2})^{2}}\sum_{k=0}^{\infty}\frac{(-1)^{k}}{k!\Gamma(k+1+\omega)}\Bigl(\frac{r}{2}\Bigr)^{2k+\omega}\Phi_{k,\varphi}^{[2]}=J_{\omega}(r).
    \end{align*}\par
    Furthermore, we show that $J_{\omega}^{[p]}$ and $\mathcal{J}_{\omega,\varphi}^{[p]}$ converge uniformly on compacts.
  \begin{proposition}[Compact convergence of series representations]\label{J-cc}
  Let $p$ satisfy $\frac{2}{p}\in\mathbb{N}$. For any $\omega\geq0,\ \varphi\in[0,2\pi)$, the series representation (\ref{series-omega}) (resp. (\ref{mathcalJ})) 
  \begin{align*}\vspace{-5pt}
      J_{\omega}^{[p]}(x)&=\frac{(\frac{|x|_{p}}{p})^{\omega}(\frac{2}{p})^{2}}
      {\Gamma(\frac{1}{p})^{2}}\sum_{k=0}^{\infty}\frac{(-1)^{k}}{\Gamma(\frac{2}{p}(k+1)
      +\omega)}\sum_{m\in\mathbb{N}_{0}^{2}\ |m|'=k}\frac{\Gamma(\frac{2m+1}{p})}{(2m)!}x^{2m},\\
      \mathcal{J}_{\omega,\varphi}^{[p]}(r)&=\frac{(\frac{2}{p})^{2}}{\Gamma(\frac{1}{p})^{2}}\sum_{k=0}^{\infty}\frac{p^{2k}(-1)^{k}}{\Gamma(\frac{2}{p}(k+1)+\omega)}\Bigl(\frac{r}{p}\Bigr)^{2k+\omega}\sum_{m\in\mathbb{N}_{0}^{2}\ |m|'=k}\frac{\Gamma(\frac{2m+1}{p})}{(2m)!}|\cos^{m_{1}}\varphi\sin^{m_{2}}\varphi|^{\frac{4}{p}}
    \end{align*}
  converge uniformly on every compact set in $\mathbb{R}^{2}$ (resp. $\mathbb{R}_{\geq0}$).
  \end{proposition}
  \begin{proof}
    Firstly, we show that inequality
    \begin{equation}\label{Gamma-ineq}
      \frac{\Gamma(n+\frac{k}{2})\Gamma(m+\frac{k}{2})}{\Gamma(n+m+k)}\leq\frac{\Gamma(\frac{k}{2})^{2}}{\Gamma(k)}
    \end{equation}
    holds for $k\in\mathbb{N},\ n,m\in\mathbb{N}_{0}$. In fact, the case $n=m=0$ is obvious, and 
    \begin{equation*}
      \hspace{-40pt}\frac{\Gamma(n+\frac{k}{2})\Gamma(m+\frac{k}{2})}{\Gamma(n+m+k)}
     =\frac{(\Pi_{l=1}^{n-1}(l+\frac{k}{2}))\frac{k}{2}\Gamma(\frac{k}{2})\cdot(\Pi_{l=1}^{m-1}(l+\frac{k}{2}))\frac{k}{2}\Gamma(\frac{k}{2})}{(\Pi_{l=1}^{n-1}(l+m+k))(m+k)(\Pi_{l=1}^{m-1}(l+k))k!}
    \leq\frac{(\frac{k}{2}\Gamma(\frac{k}{2}))^{2}}{(m+k)k!}\leq\frac{\Gamma(\frac{k}{2})^{2}}{4\Gamma(k)},
    \end{equation*}
    holds for the cases $n,m\in\mathbb{N}$, while 
    \begin{equation*}
      \frac{\Gamma(\frac{k}{2})\Gamma(m+\frac{k}{2})}{\Gamma(m+k)}=\frac{\Gamma(\frac{k}{2})(\Pi_{l=1}^{m-1}(l+\frac{k}{2}))\frac{k}{2}\Gamma(\frac{k}{2})}{(\Pi_{l=1}^{m-1}(l+k))k!}\leq\frac{\Gamma(\frac{k}{2})^{2}}{2\Gamma(k)}
    \end{equation*}
    holds for the cases $n=0,m\in\mathbb{N}$ (similarly for the cases $n\in\mathbb{N},m=0$). \vspace{3pt}\par
    In addition, we also consider the fact that $\Gamma(n+\omega)=(\Pi_{l=1}^{n-1}(l+\omega))\omega\Gamma(\omega)\geq\Gamma(n)\omega\Gamma(\omega)$ holds for $\omega>0$ and define the following. \vspace{-5pt}
    \begin{equation*}\vspace{-5pt}
      C(\omega):=\frac{1}{\omega\Gamma(\omega)}\quad\text{for }\omega>0,\quad 1\quad\text{for }\omega=0.
    \end{equation*}\par\vspace{5pt}
    Then, for $p$ satisying $\frac{2}{p}\in\mathbb{N},\ \omega\geq0,\ x\in K\subset\mathbb{R}^{2}$ ($K$: an arbitrary compact set), the display(\ref{series-omega})
    \begin{align*}
      J_{\omega}^{[p]}(x)&=\Bigl(\frac{|x|_{p}}{p}\Bigr)^{\omega}\Bigl(\frac{2}{p\Gamma(\frac{1}{p})}\Bigr)^{2}\sum_{k=0}^{\infty}\frac{(-1)^{k}}{\Gamma(\frac{2}{p}(k+1)+\omega)}\sum_{m\in\mathbb{N}_{0}^{2}\ |m|'=k}\frac{\Gamma(\frac{2m+1}{p})}{(2m)!}x^{2m}\\
      &=\sum_{m_{1}=0}^{\infty}\sum_{m_{2}=0}^{\infty}\Bigl(\frac{|x|_{p}}{p}\Bigr)^{\omega}\Bigl(\frac{2}{p\Gamma(\frac{1}{p})}\Bigr)^{2}\frac{(-1)^{m_{1}+m_{2}}}{\Gamma(\frac{2}{p}(m_{1}+m_{2}+1)+\omega)}\frac{\Gamma(\frac{2m+1}{p})}{(2m)!}x^{2m}
    \end{align*}
    uniformly converges on $K$ (that is, uniformly converge on every compact set in $\mathbb{R}^{2}$). In fact, from the inequality (\ref{Gamma-ineq}), since it can be evaluated from above as 
    \begin{equation}\label{C-omega}
      \frac{\Gamma(\frac{2}{p}m_{1}+\frac{1}{p})\Gamma(\frac{2}{p}m_{2}+\frac{1}{p})}{\Gamma(\frac{2}{p}(m_{1}+m_{2}+1)+\omega)}\leq C(\omega)\frac{\Gamma(\frac{2}{p}m_{1}+\frac{1}{p})\Gamma(\frac{2}{p}m_{2}+\frac{1}{p})}{\Gamma(\frac{2}{p}(m_{1}+m_{2}+1))}\leq C(\omega)\frac{\Gamma(\frac{1}{p})^{2}}{\Gamma(\frac{2}{p})},
    \end{equation}
    the following holds, noting the Cauchy product.
    \begin{align*}
      \sum_{m_{1}=0}^{\infty}\sum_{m_{2}=0}^{\infty}\Bigl|\Bigl(\frac{|x|_{p}}{p}\Bigr)^{\omega}\Bigl(\frac{2}{p\Gamma(\frac{1}{p})}\Bigr)^{2}&\frac{(-1)^{m_{1}+m_{2}}}{\Gamma(\frac{2}{p}(m_{1}+m_{2}+1)+\omega)}\frac{\Gamma(\frac{2m+1}{p})}{(2m)!}x^{2m}\Bigr|\\
      &\leq\sum_{m_{1}=0}^{\infty}\sum_{m_{2}=0}^{\infty}\Bigl(\frac{|x|_{p}}{p}\Bigr)^{\omega}\Bigl(\frac{2}{p\Gamma(\frac{1}{p})}\Bigr)^{2}\frac{|x_{1}|^{2m_{1}}|x_{2}|^{2m_{2}}}{(2m_{1})!(2m_{2})!}C(\omega)\frac{\Gamma(\frac{1}{p})^{2}}{\Gamma(\frac{2}{p})}\\
      &=\frac{4C(\omega)|x|_{p}^{\omega}}{p^{\omega+2}\Gamma(\frac{2}{p})}\Bigl(\sum_{m_{1}=0}^{\infty}\frac{|x_{1}|^{2m_{1}}}{(2m_{1})!}\Bigr)\Bigl(\sum_{m_{2}=0}^{\infty}\frac{|x_{2}|^{2m_{2}}}{(2m_{2})!}\Bigr)\\
      &\leq\frac{4C(\omega)|x|_{p}^{\omega}}{p^{\omega+2}\Gamma(\frac{2}{p})}e^{|x_{1}|+|x_{2}|}<+\infty.
    \end{align*}\par
    By the following obtained from the similar discussion, it is also shown that the series representation (\ref{mathcalJ}) converges uniformly on every compact set in $\mathbb{R}_{\geq0}$.
    \begin{align*}
      \sum_{m_{1}=0}^{\infty}\sum_{m_{2}=0}^{\infty}\Bigl|\Bigl(\frac{r}{p}\Bigr)^{\omega}\Bigl(\frac{2}{p\Gamma(\frac{1}{p})}\Bigr)^{2}\frac{(-1)^{m_{1}+m_{2}}}{\Gamma(\frac{2}{p}(m_{1}+m_{2}+1)+\omega)}&\frac{\Gamma(\frac{2m+1}{p})}{(2m)!}(r\cos^{\frac{2}{p}}\varphi)^{2m_{1}}(r\sin^{\frac{2}{p}}\varphi)^{2m_{2}}\Bigr|\\
  &\hspace{-95pt}\leq\frac{4C(\omega)r^{\omega}}{p^{\omega+2}\Gamma(\frac{2}{p})}e^{r(|\cos\varphi|^{\frac{2}{p}}+|\sin\varphi|^{\frac{2}{p}})}<+\infty.
    \end{align*}
  \end{proof}\par
  Now, we follow the method deriving the integral formula for $J_{\omega}$ (\cite{Stein-1971}, Lemma 4.13)
  \begin{equation}\label{Ire-J}
    J_{\omega+\gamma}(r)=\frac{r^{\gamma}}{2^{\gamma-1}\Gamma(\gamma)}\int_{0}^{1}
    J_{\omega}(\tau r)\tau^{\omega+1}(1-\tau^{2})^{\gamma-1}d\tau\qquad\text{for }\omega>-\frac{1}{2},\ \gamma>0,\ r>0
  \end{equation}
  and derive the counterparts for $J_{\omega}^{[p]}$ and $\mathcal{J}_{\omega,\varphi}^{[p]}$ as follows. \vspace{10pt}\par
  Let $p$ satisfy $\frac{2}{p}\in\mathbb{N}$. For $\omega\geq0,\ \gamma>0,\ x\in\mathbb{R}^{2}\setminus{\{0\}}$, from the series representation (\ref{series-omega}), based on the similar discussions as in the proof of Proposition \ref{J-cc}, we express it as
   \begin{align*}
    \int_{0}^{1}J_{\omega}^{[p]}&(\tau x)\tau^{(p-1)\omega+1}(1-\tau^{p})^{\gamma-1}d\tau\\
    &=\int_{0}^{1}\Bigl(\frac{(\frac{2}{p})^{2}|\tau x|_{p}^{\omega}}{p^{\omega}\Gamma(\frac{1}{p})^{2}}\sum_{k=0}^{\infty}\frac{(-1)^{k}}{\Gamma(\frac{2}{p}(k+1)+\omega)}\sum_{m\in\mathbb{N}_{0}^{2}\ |m|'=k}\frac{\Gamma(\frac{2m+1}{p})}{(2m)!}\tau^{2k}x^{2m}\Bigr)\tau^{(p-1)\omega+1}(1-\tau^{p})^{\gamma-1}d\tau\\
    &=\frac{(\frac{2}{p})^{2}|x|_{p}^{\omega}}{p^{\omega}\Gamma(\frac{1}{p})^{2}}
    \sum_{k=0}^{\infty}\frac{(-1)^{k}}{\Gamma(\frac{2}{p}(k+1)+\omega)}\sum_{m\in\mathbb{N}_{0}^{2}\ |m|'=k}\frac{\Gamma(\frac{2m+1}{p})}{(2m)!}x^{2m}\int_{0}^{1}\tau^{2(k+1)+p\omega-1}(1-\tau^{p})^{\gamma-1}d\tau\\
    &=\frac{(\frac{2}{p})^{2}|x|_{p}^{\omega}}{p^{\omega}\Gamma(\frac{1}{p})^{2}}\sum_{k=0}^{\infty}\frac{(-1)^{k}}{\Gamma(\frac{2}{p}(k+1)+\omega)}\sum_{m\in\mathbb{N}_{0}^{2}\ |m|'=k}\frac{\Gamma(\frac{2m+1}{p})}{(2m)!}x^{2m}\Bigl(\frac{\Gamma(\frac{2(k+1)}{p}+\omega)\Gamma(\gamma)}{\ p\Gamma(\frac{2(k+1)}{p}+(\omega+\gamma))}\Bigr)\\
    &=\frac{\Gamma(\gamma)p^{\gamma-1}}{|x|_{p}^{\gamma}}\frac{(\frac{2}{p})^{2}|x|_{p}^{\omega+\gamma}}{p^{\omega+\gamma}\Gamma(\frac{1}{p})^{2}}\sum_{k=0}^{\infty}\frac{(-1)^{k}}{\Gamma(\frac{2}{p}(k+1)+(\omega+\gamma))}\sum_{m\in\mathbb{N}_{0}^{2}\ |m|'=k}\frac{\Gamma(\frac{2m+1}{p})}{(2m)!}x^{2m}\\
    &=\frac{\Gamma(\gamma)p^{\gamma-1}}{|x|_{p}^{\gamma}}J_{\omega+\gamma}^{[p]}(x).
  \end{align*}\par
  The term-by-term integrability in the second to third lines is guaranteed by the special case of Fubini's theorem, since 
  \begin{equation*}
    \hspace{-20pt}\sum_{m\in\mathbb{N}_{0}^{2}}\int_{0}^{1}\Bigl|\frac{(\frac{2}{p})^{2}|\tau x|_{p}^{\omega}(-1)^{|m|'}\Gamma(\frac{2m+1}{p})\tau^{2|m|'}x^{2m}\tau^{(p-1)\omega+1}(1-\tau^{p})^{\gamma-1}}{p^{\omega}\Gamma(\frac{1}{p})^{2}\Gamma(\frac{2}{p}(|m|'+1)+\omega)(2m)!}\Bigr|d\tau\leq\frac{4C(\omega)|x|_{p}^{\omega}B(\frac{1}{p},\gamma)e^{|x_{1}|+|x_{2}|}}{p^{\omega+3}\Gamma(\frac{2}{p})}<+\infty
  \end{equation*}
  holds from Proposition \ref{J-cc} and (\ref{C-omega}). As for the equal sign in the third to fourth lines, we used the formula to display the beta function from the gamma function only, by expressing the variable transformation as follows.
  \begin{equation*}
  \int_{0}^{1}\tau^{2(k+1)+p\omega-1}(1-\tau^{p})^{\gamma-1}d\tau=\int_{0}^{1}r^{\frac{2}{p}(k+1)+\omega-\frac{1}{p}}(1-r)^{\gamma-1}p^{-1}r^{\frac{1}{p}-1}dr=p^{-1}B\Bigl(\frac{2(k+1)}{p}+\omega,\gamma\Bigr).
  \end{equation*}
  \par
  By transforming the equation for the one-variable function $J_{\omega,\varphi}^{[p]}$ with fixed distorted angle $\varphi\in[0,2\pi)$ from the similar discussions, we obtain the following two formulas as the desired corresponding expressions of the integral formula (\ref{Ire-J}).
  \begin{equation}
    J_{\omega+\gamma}^{[p]}(x)=\frac{|x|_{p}^{\gamma}}{p^{\gamma-1}\Gamma(\gamma)}\int_{0}^{1}J_{\omega}^{[p]}(\tau x)\tau^{(p-1)\omega+1}(1-\tau^{p})^{\gamma-1}d\tau\quad\text{for }\omega\geq0,\ \gamma>0,\ x\in\mathbb{R}^{2}\setminus{\{0\}},\label{omega_J}
  \end{equation}\vspace{-5pt}
  \begin{equation}
    \hspace{-55pt}\mathcal{J}_{\omega+\gamma,\varphi}^{[p]}(r)=\frac{r^{\gamma}}{p^{\gamma-1}\Gamma(\gamma)}\int_{0}^{1}\mathcal{J}_{\omega,\varphi}^{[p]}(\tau r)\tau^{(p-1)\omega+1}(1-\tau^{p})^{\gamma-1}d\tau\quad\text{for }\omega\geq0,\ \gamma>0,\ r>0.\label{mathcalJ-int-rec}
  \end{equation}
  
  \subsection{\normalsize{The Erd\'{e}lyi-Kober operator and a differential formula for $\mathcal{J}_{\omega,\varphi}^{[p]}$}}
  \hspace{13pt}Firstly, we introduce \textit{the Erd\'{e}lyi-Kober type fractional integrals} (\cite{FDE}, (2.6.1)), 
  \begin{equation*}\label{E-K-op.}
    \hspace{-30pt}(I_{a+;p,\eta}^{\alpha})f(r):=\frac{pr^{-p(\alpha+\eta)}}{\Gamma(\alpha)}\int_{a}^{r}\frac{t^{p(\eta+1)-1}f(t)}{(r^{p}-t^{p})^{1-\alpha}}dt\ \Bigl(=\frac{1}{\Gamma(\alpha)}\int_{(\frac{a}{r})^{p}}^{1}\frac{t^{\eta}f(t^{\frac{1}{p}}r)}{(1-t)^{1-\alpha}}dt\Bigr)\quad\text{for }0\leq a<r,
  \end{equation*}
   for $\alpha,\eta\in\mathbb{C}$ satisfying $\mathrm{Re}(\alpha)>0$. Note that $\alpha$ is an order. In particular, for $a=0$, since this can be expressed as 
   \begin{equation}\label{E-K-op. a=0}
    (I_{0+;p,\eta}^{\alpha})f(r)=\frac{pr^{-p(\alpha+\eta)}}{\Gamma(\alpha)}\int_{0}^{1}\frac{r^{p(\eta+1)-1}\tau^{p(\eta+1)-1}f(\tau r)}{r^{p(1-\alpha)}(1-\tau^{p})^{1-\alpha}}rd\tau=\frac{p}{\Gamma(\alpha)}\int_{0}^{1}\frac{\tau^{p(\eta+1)-1}f(\tau r)}{(1-\tau^{p})^{1-\alpha}}d\tau,
  \end{equation}
   by applying it to $f=\mathcal{J}_{\omega,\varphi}^{[p]},\ \alpha=\gamma,\ \eta=(1-\frac{1}{p})\omega+\frac{2}{p}-1$, from the integral formula (\ref{mathcalJ-int-rec}), we obtain 
  \begin{equation*}
    (I_{0+;p,(1-\frac{1}{p})\omega+\frac{2}{p}-1}^{\gamma})\mathcal{J}_{\omega,\varphi}^{[p]}(r)=\frac{p}{\Gamma(\gamma)}\int_{0}^{1}\tau^{(p-1)\omega+1}(1-\tau^{p})^{\gamma-1}\mathcal{J}_{\omega,\varphi}^{[p]}(\tau r)d\tau=\frac{p}{\Gamma(\gamma)}\Bigl(\frac{\Gamma(\gamma)p^{\gamma-1}}{r^{\gamma}}\mathcal{J}_{\omega+\gamma,\varphi}^{[p]}(r)\Bigr),
  \end{equation*}
   that is, 
  \begin{equation}\label{E-K-intJ}
    (I_{0+;p,(1-\frac{1}{p})\omega+\frac{2}{p}-1}^{\gamma})\mathcal{J}_{\omega,\varphi}^{[p]}(r)=\Bigl(\frac{p}{r}\Bigr)^{\gamma}\mathcal{J}_{\omega+\gamma,\varphi}^{[p]}(r).
  \end{equation}
  \begin{remark}
  The form of the equality (\ref{E-K-intJ}) gives justification to the following equality for $J_{\omega}^{[p]}$ (\ref{A-I-intJ}), which is formally derived by using the Erd\'{e}lyi-Kober type fractional integrals operator of multiple variables (see \cite{AomotoIguchi}) 
  \begin{equation*}\label{2-E-K}
    P_{\kappa}(\eta,\alpha)f(x)=\frac{1}{\Gamma(\alpha)}\int_{0}^{1}\frac{t^{\eta}f(t^{\kappa}x)}{(1-t)^{1-\alpha}}dt\qquad\text{for }x\in\mathbb{R}^{2}
  \end{equation*}
  and the integral formula (\ref{omega_J}), without considering the regularity conditions and the range of some symbols (that is, (\ref{A-I-intJ}) corresponds to the version of multiple variables of (\ref{E-K-intJ})). \vspace{-5pt}  
  \begin{align}
    \Bigl(\frac{p}{|x|_{p}}\Bigr)^{\gamma}J_{\omega+\gamma}^{[p]}(x)&=\Bigl(\frac{p}{|x|_{p}}\Bigr)^{\gamma}\frac{|x|_{p}^{\gamma}}{p^{\gamma-1}\Gamma(\gamma)}\int_{0}^{1}J_{\omega}^{[p]}(\tau x)\tau^{(p-1)\omega+1}(1-\tau^{p})^{\gamma-1}d\tau\notag\\
    &=\frac{p}{\Gamma(\gamma)}\int_{0}^{1}J_{\omega}^{[p]}(t^{\frac{1}{p}}x)t^{(1-\frac{1}{p})\omega+\frac{1}{p}}(1-t)^{\gamma-1}p^{-1}t^{\frac{1}{p}-1}dt\notag\\
    &=\frac{1}{\Gamma(\gamma)}\int_{0}^{1}\frac{t^{(1-\frac{1}{p})\omega+\frac{2}{p}-1}
    J_{\omega}^{[p]}(t^{\frac{1}{p}}x)}{(1-t)^{1-\gamma}}dt\notag\\
    &=P_{\frac{1}{p}}\Bigl(\Bigl(1-\frac{1}{p}\Bigr)\omega+\frac{2}{p}-1,\gamma\Bigr)J_{\omega}^{[p]}(x).\label{A-I-intJ}
  \end{align}
  \end{remark}\vspace{10pt}
  Furthermore, for $\alpha\in\mathbb{C}\setminus{\{0\}}$ satisfying $\mathrm{Re}(\alpha)\geq0$, $n:=[\mathrm{Re}(\alpha)]+1,\ p>0$ and $\eta\in\mathbb{C}$, we introduce \textit{the Erd\'{e}lyi-Kober type fractional derivatives} (\cite{FDE}, (2.6.29)) 
  \begin{equation}\label{E-K-op.2}
    (D_{a+;p,\eta}^{\alpha})f(r):=r^{-p\eta}\Bigl(\frac{1}{pr^{p-1}}\frac{d}{dr}\Bigr)^{n}r^{p(n+\eta)}(I_{a+;p,\eta+\alpha}^{n-\alpha})f(r)\qquad\text{for }0\leq a<r.
  \end{equation}\par
  Then, the following important property holds between the two operators (\cite{FDE}, (2.6.43)).
  \begin{equation}\label{DI}
    (D_{a+;p,\eta}^{\alpha}I_{a+;p,\eta}^{\alpha})f(r)=f(r).
  \end{equation}
  \par\vspace{10pt}
  Now, let $p$ satisfy $\frac{2}{p}\in\mathbb{N}$, $\omega\geq0,\ 0\leq\varphi<2\pi$, $a=0$, $\alpha=\gamma\in(0,1)$ and $\eta=(1-\frac{1}{p})\omega+\frac{2}{p}-1$, then
  the Erd\'{e}lyi-Kober type fractional derivative (\ref{E-K-op.2}) of $(\frac{p}{r})^{\gamma}\mathcal{J}_{\omega+\gamma,\varphi}^{[p]}(r)$ can be expressed as 
  \begin{equation*}
    D^{\gamma}_{0+;p, (1-\frac{1}{p})\omega+\frac{2}{p}-1}\Bigl[\Bigl(\frac{p}{r}\Bigr)^{\gamma}\mathcal{J}_{\omega+\gamma,\varphi}^{[p]}(r)\Bigr]=\frac{r^{(1-p)\omega+p-2}}{pr^{p-1}}\frac{d}{dr}r^{(p-1)\omega+2}\Bigl(I^{1-\gamma}_{0+;p,(1-\frac{1}{p})\omega+\frac{2}{p}-1+\gamma}\Bigl[\Bigl(\frac{p}{r}\Bigr)^{\gamma}\mathcal{J}_{\omega+\gamma,\varphi}^{[p]}(r)\Bigr]\Bigr).
  \end{equation*}
  Combining this with (\ref{E-K-intJ}), (\ref{DI}) and (\ref{E-K-op. a=0}), we obtain
  \begin{align*}
    \mathcal{J}_{\omega,\varphi}^{[p]}(r)&=D^{\gamma}_{0+;p, (1-\frac{1}{p})\omega+\frac{2}{p}-1}\Bigl[\Bigl(\frac{p}{r}\Bigr)^{\gamma}\mathcal{J}_{\omega+\gamma,\varphi}^{[p]}(r)\Bigr]\\
    &=\frac{r^{(1-p)\omega-1}}{p}\frac{d}{dr}r^{(p-1)\omega+2}\Bigl(\frac{p}{\Gamma(1-\gamma)}\int_{0}^{1}\frac{\tau^{p((1-\frac{1}{p})\omega+\frac{2}{p}-1+\gamma+1)-1}}{(1-\tau^{p})^{1-(1-\gamma)}}\Bigl(\frac{p}{\tau r}\Bigr)^{\gamma}\mathcal{J}_{\omega+\gamma,\varphi}^{[p]}(\tau r)d\tau\Bigr)\\
    &=\frac{1}{r^{1+(p-1)\omega}}\frac{d}{dr}r^{(p-1)\omega+2-\gamma}\frac{p^{\gamma}}{\Gamma(1-\gamma)}\int_{0}^{1}\frac{\tau^{(p-1)\omega+2+p\gamma-1-\gamma}}{(1-\tau^{p})^{1-(1-\gamma)}}\mathcal{J}_{\omega+\gamma,\varphi}^{[p]}(\tau r)d\tau\\
    &=\frac{1}{r^{1+(p-1)\omega}}\frac{d}{dr}r^{1+(p-1)\omega}\frac{r^{1-\gamma}}{p^{(1-\gamma)-1}\Gamma(1-\gamma)}\int_{0}^{1}\frac{\tau^{(p-1)(\omega+\gamma)+1}}{(1-\tau^{p})^{1-(1-\gamma)}}\mathcal{J}_{\omega+\gamma,\varphi}^{[p]}(\tau r)d\tau\\
    &=\frac{1}{r^{1+(p-1)\omega}}\frac{d}{dr}r^{1+(p-1)\omega}\mathcal{J}_{(\omega+\gamma)+(1-\gamma)}^{[p]}(r).
  \end{align*}
  Note the use of the integral formula (\ref{mathcalJ-int-rec}) in the last line. Thus, it is clear that differential formula of decreasing order holds as follows. 
  \begin{proposition}[\textit{Differential formulas}]\label{df-p}
  \itshape{Let $p>0$ such that $\frac{2}{p}\in\mathbb{N}$, then the following holds.}
  \begin{equation}\label{df-p1}
    \frac{d}{dr}r^{1+(p-1)\omega}\mathcal{J}_{\omega+1,\varphi}^{[p]}(r)=r^{1+(p-1)\omega}\mathcal{J}_{\omega,\varphi}^{[p]}(r)\qquad\text{for }\omega\geq0,\ 0\leq\varphi<2\pi.
  \end{equation}
   In particular, due to $\mathcal{J}_{\omega,\varphi}^{[2]}=J_{\omega}$, this result includes the following differential formula for the Bessel functions. 
  \begin{equation}\label{df-1}
    \frac{d}{dr}r^{\omega+1}J_{\omega+1}(r)=r^{\omega+1}J_{\omega}(r).
  \end{equation}
  \end{proposition}
 
  \par\vspace{7pt}
  \begin{remark}
  While it was difficult to deductively predict the form of the differential formula (\ref{df-p1}), we could derive the general form of the formula based on the results for some special orders $\omega$ obtained from the Erd\'{e}lyi-Kober type fractional derivatives. Then, as long as we can guess the form of this general formula, we can show that the formula holds by following the well-known proof (see \cite{Watson}, p45) of the differential formula (\ref{df-1}) by using the series representation of the Bessel functions\vspace{-5pt}
  \begin{equation*}\label{J-series}
    J_{\omega}(r)=\sum_{k=0}^{\infty}\frac{(-1)^{k}}{k!\Gamma(k+1+\omega)}\Bigl(\frac{r}{2}\Bigr)^{2k+\omega}.
  \end{equation*}
  In fact, from the series representation (\ref{mathcalJ-1})
  \begin{equation*}
      \mathcal{J}_{\omega,\varphi}^{[p]}(r)=\Bigl(\frac{2}{p}\Bigr)^{2+\omega}\frac{\pi}{\Gamma(\frac{1}{p})^{2}}\sum_{k=0}^{\infty}\frac{(-1)^{k}}{k!\Gamma(\frac{2}{p}(k+1)+\omega)}\Bigl(\frac{r}{2}\Bigr)^{2k+\omega}\Phi_{k,\varphi}^{[p]},\vspace{-8pt}
    \end{equation*}
    it is shown as follows.
  \begin{align*}
    \frac{d}{dr}r^{1+(p-1)\omega}\mathcal{J}_{\omega+1,\varphi}^{[p]}(r)&=\Bigl(\frac{2}{p}\Bigr)^{2+\omega+1}\frac{\pi}{\Gamma(\frac{1}{p})^{2}}\frac{d}{dr}\sum_{k=0}^{\infty}\frac{(-1)^{k}}{k!\Gamma(\frac{2}{p}(k+1)+\omega+1)}\Bigl(\frac{r}{2}\Bigr)^{2(k+\frac{p}{2}\omega+1)}2^{1+(p-1)\omega}\Phi_{k,\varphi}^{[p]}\\
  &=\Bigl(\frac{2}{p}\Bigr)^{2+\omega+1}\frac{\pi}{\Gamma(\frac{1}{p})^{2}}\sum_{k=0}^{\infty}\frac{(-1)^{k}2^{1+(p-1)\omega}\cdot2(k+\frac{p}{2}\omega+1)}{k!\Gamma(\frac{2}{p}(k+1)+\omega+1)}\Bigl(\frac{r}{2}\Bigr)^{2(k+\frac{p}{2}\omega)+1}\frac{1}{2}\Phi_{k,\varphi}^{[p]}\\
  &=r^{1+(p-1)\omega}\Bigl(\frac{2}{p}\Bigr)^{2+\omega}\frac{\pi}{\Gamma(\frac{1}{p})^{2}}\sum_{k=0}^{\infty}\frac{(-1)^{k}}{k!}\frac{\frac{2}{p}(k+1)+\omega}{(\frac{2}{p}(k+1)+\omega)\Gamma(\frac{2}{p}(k+1)+\omega)}\Bigl(\frac{r}{2}\Bigr)^{2k+\omega}\Phi_{k,\varphi}^{[p]}\\
  &=r^{1+(p-1)\omega}\Bigl(\frac{2}{p}\Bigr)^{2+\omega}\frac{\pi}{\Gamma(\frac{1}{p})^{2}}\sum_{k=0}^{\infty}\frac{(-1)^{k}}{k!\Gamma(\frac{2}{p}(k+1)+\omega)}\Bigl(\frac{r}{2}\Bigr)^{2k+\omega}\Phi_{k,\varphi}^{[p]}\\
  &=r^{1+(p-1)\omega}\mathcal{J}_{\omega,\varphi}^{[p]}(r).
  \end{align*}
  \end{remark}
  
  \subsection{\normalsize{Proof of Theorem \ref{GHI}}}
  \hspace{13pt}Let $p>0$, a function $F$ be \textit{$p$-radial} (that is, there exists a function on $\mathbb{R}_{\geq0}$ $\phi$ satisfying $F(x)=\phi(|x|_{p})$ in any $x\in\mathbb{R}^{2}$) and be integrable on $\mathbb{R}^{2}$. Then, the generalization for $p$ of \textit{the Hankel transform of order zero} (\cite{K1}, (2.2))
  \begin{equation*}
    \hat{F}(x)=p\Gamma(\frac{1}{p})^{2}\int_{0}^{\infty}J_{0}^{[p]}(2\pi tx)\phi(t)t\ dt\qquad\text{for }x\in\mathbb{R}^{2}
  \end{equation*}
  holds, and it can be expressed as 
  \begin{equation*}
    \sum_{n\in\mathbb{Z}^{2}}F(n)=\sum_{n\in\mathbb{Z}^{2}}\hat{F}(n)=\sum_{n\in\mathbb{Z}^{2}}\Bigl(p\Gamma(\frac{1}{p})^{2}\int_{0}^{\infty}J_{0}^{[p]}(2\pi tn)\phi(t)t\ dt\Bigr),
  \end{equation*}
  from \textit{the Poisson summation formula} (an equality for periodization of integrable functions; \cite{Stein-1971}, p251, Theorem 2.4). In particular, let $F$ be an indicator function on the $p$-circle of radius $r$ (that is, $F(x)=1\ (|x|_{p}<r),\ 0\ (|x|_{p}\geq r)$) and $N_{p}(r)$ be the number of lattice points inside this $p$-circle, then
  \begin{align*}
    N_{p}(r)&=p\Gamma(\frac{1}{p})^{2}\Bigl(\int_{0}^{r}\frac{(\frac{2}{p})^{2}}{\Gamma(\frac{2}{p})}t\ dt+\sum_{n\in\mathbb{Z}^{2}\setminus{\{0\}}}\Bigl(\int_{0}^{r}J_{0}^{[p]}(2\pi tn)t\ dt\Bigr)\Bigr)\\
    &=\frac{2}{p}\frac{\Gamma(\frac{1}{p})^{2}}{\Gamma(\frac{2}{p})}r^{2}+p\Gamma(\frac{1}{p})^{2}\sum_{n\in\mathbb{Z}^{2}\setminus{\{0\}}}\Bigl(\int_{0}^{r}J_{0}^{[p]}(2\pi tn)t\ dt\Bigr)
  \end{align*}
  holds. Note that $J_{0}^{[p]}(0)=(\frac{2}{p})^{2}/\Gamma(\frac{2}{p})$ is used for the first term on the right-hand side. 
  \vspace{5pt}\par
  Now, we consider in particular for $p$ satisfying $(2/p)\in\mathbb{N}$. For the finite set $\mathcal{A}_{s}^{[p]}$(\ref{mathcalA}), that is, the set consisting of distorted angles $\varphi$ corresponding to lattice points on the $p$-circle of radius $s^{1/p}\ (\geq1)$
  \begin{equation*}
    \mathcal{A}_{s}^{[p]}=\{\varphi\in[0,2\pi)\ |\ (\mathrm{sgn}(\cos\varphi)s^{\frac{1}{p}}|\cos\varphi|^{\frac{2}{p}},\mathrm{sgn}(\sin\varphi)s^{\frac{1}{p}}|\sin\varphi|^{\frac{2}{p}})\in\mathbb{Z}^{2}\}
    \end{equation*}
  ($|n_{1}|^{p}+|n_{2}|^{p}=s$, that is, $|n|_{p}=s^{\frac{1}{p}}$), and the functions $g$ and $\mathcal{G}$ satisfying $g(n)=\mathcal{G}_{\varphi}(s^{\frac{1}{p}})$ under this distorted polar coordinate system, the following holds with the counting measure $\mu$. 
  \begin{equation*}
    \hspace{-10pt}\sum_{n\in\mathbb{Z}^{2}\setminus{\{0\}}}g(n)=\lim_{T\to\infty}\sum\nolimits_{1\leq|n|_{p}^{p}\leq T}g(n)=\lim_{T\to\infty}\int_{1}^{T}\sum\nolimits_{\varphi\in\mathcal{A}_{s}^{[p]}}\mathcal{G}_{\varphi}(s^{\frac{1}{p}})d\mu(s)=\int_{1}^{\infty}\sum\nolimits_{\varphi\in\mathcal{A}_{s}^{[p]}}\mathcal{G}_{\varphi}(s^{\frac{1}{p}})d\mu(s).
  \end{equation*}\par
  In addition to the above, from (\ref{P_p}) and the special case of Proposition \ref{df-p} (differential formula), that is, $\int_{0}^{r}\tau\mathcal{J}_{0,\varphi}^{[p]}(\tau)d\tau=r\mathcal{J}_{1,\varphi}^{[p]}(r)$, 
  \begin{align*}
    P_{p}(r)&=p\Gamma(\frac{1}{p})^{2}\sum_{n\in\mathbb{Z}^{2}\setminus{\{0\}}}\Bigl(\int_{0}^{r}J_{0}^{[p]}(2\pi tn)t\ dt\Bigr)\\
    &=p\Gamma(\frac{1}{p})^{2}\int_{1}^{\infty}\sum\nolimits_{\varphi\in\mathcal{A}_{s}^{[p]}}\Bigl(\int_{0}^{r}(2\pi ts^{\frac{1}{p}})\mathcal{J}_{0,\varphi}^{[p]}(2\pi ts^{\frac{1}{p}})dt\Bigr)\frac{d\mu(s)}{2\pi s^{\frac{1}{p}}}\\
    &=p\Gamma(\frac{1}{p})^{2}\int_{1}^{\infty}\sum\nolimits_{\varphi\in\mathcal{A}_{s}^{[p]}}\Bigl(\int_{0}^{2\pi rs^{\frac{1}{p}}}\tau\mathcal{J}_{0,\varphi}^{[p]}(\tau)\frac{d\tau}{2\pi s^{\frac{1}{p}}}\Bigr)\frac{d\mu(s)}{2\pi s^{\frac{1}{p}}}\\
    &=p\Gamma(\frac{1}{p})^{2}\int_{1}^{\infty}\sum\nolimits_{\varphi\in\mathcal{A}_{s}^{[p]}}\Bigl(2\pi rs^{\frac{1}{p}}\mathcal{J}_{1,\varphi}^{[p]}(2\pi s^{\frac{1}{p}}r)\Bigr)\frac{d\mu(s)}{(2\pi s^{\frac{1}{p}})^{2}}
  \end{align*}\vspace{-12pt}\\
  holds. Hence, we obtain the main result formula in conclusion.   
  \begin{equation}\label{HI-p}
    P_{p}(r)=\frac{p\Gamma(\frac{1}{p})^{2}}{2\pi}r\int_{1}^{\infty}\frac{1}{s^{\frac{1}{p}}}\Bigl(\sum\nolimits_{\varphi\in\mathcal{A}_{s}^{[p]}}\mathcal{J}_{1,\varphi}^{[p]}(2\pi s^{\frac{1}{p}}r)\Bigr)d\mu(s). 
  \end{equation}\vspace{5pt}
  \begin{remark}
  In particular, for $p=2$, we define $R(k):=\#\{n\in\mathbb{Z}^{2}|\ |n|^{2}=k\}$ and note $s\in\mathbb{N}$ and $\mathcal{J}_{\omega,\varphi}^{[2]}(r)=J_{\omega}(r)$ (distorted angles invariance), then 
  \begin{equation*}
    P_{2}(r)=r\sum_{k=1}^{\infty}\frac{1}{k^{\frac{1}{2}}}J_{1}(2\pi k^{\frac{1}{2}}r)\Bigl(\sum\nolimits_{\varphi\in\mathcal{A}_{k}^{[2]}}1\Bigr)
    =r\sum_{k=1}^{\infty}\frac{R(k)}{k^{\frac{1}{2}}}J_{1}(2\pi k^{\frac{1}{2}}r)
  \end{equation*}
  can be expressed. Thus, (\ref{HI-p}) is a generalization for $p$ to Hardy's identity (\ref{HI}). 
  \end{remark}
  %%%%%%%%%%%%%%%%%%%%%%% =============================================================================
  \section{Concluding remarks}
    \hspace{13pt}In this paper, the generalized formula (\ref{df-p1}), which generalize the differential formula of decreasing order (\ref{df-1}), can be derived by the Erd\'{e}lyi-Kober fractional derivative. On the other hand, as for the derivation of another important formula of Bessel functions, the differential formula of increasing order
    \begin{equation}\label{df-2}
      \frac{d}{dr}\Bigl(\frac{J_{\omega}(r)}{r^{\omega}}\Bigr)=-\frac{J_{\omega+1}(r)}{r^{\omega}}, 
    \end{equation}
    it is found that the well-known proof method via the series representation cannot be applied in general due to the complication of the part $\Phi_{k,\varphi}^{[p]}$ consisting of the distorted angle $\varphi$. In view of the above, we plan to derive a generalized formula of the differential formula (\ref{df-2}) by using the Erd\'{e}lyi-Kober operator, which is strongly related to $J_{\omega}^{[p]},\ \mathcal{J}_{\omega,\varphi}^{[p]}$. Unlike (\ref{df-1}), which consists of integer-order derivatives, even if we succeed in deriving the formula, it will probably be in the form of a fractional-order derivative formula. However, by combining the obtained formula with (\ref{df-1}), it is expected to derive a fractional differential equation whose solution is $\mathcal{J}_{\omega,\varphi}^{[p]}$. \vspace{15pt}\par
    In what follows, we discuss the relevance of the results of this paper to our previous studies. In paper \cite{K1}, as a partial generalization of the method by S. Kuratsubo and E. Nakai in 2022 (see \cite{Kuratsubo-2022}), we defined two functions
  \begin{equation}\label{D2}
    D_{\beta}^{[p]}(s:x):=\frac{1}{\Gamma(\beta+1)}\sum_{|m|_{p}^{p}<s}(s-|m|_{p}^{p})^{\beta}e^{2\pi ix\cdot m},\quad 
    \mathcal{D}_{\beta}^{[p]}(s:x):=\frac{1}{\Gamma(\beta+1)}\int_{|\xi|_{p}^{p}<s}
    (s-|\xi|_{p}^{p})^{\beta}e^{2\pi ix\cdot \xi}d\xi,
  \end{equation} 
 and established the following theorem (see also \cite{K2}, Theorem 3.1).
  \begin{theorem}[\textit{Astroid-type versions of Theorem 1.3 of \cite{K1}}]\label{thm_At}
  \itshape{Let $p>0$ such that $\frac{2}{p}\in\mathbb{N}\setminus{\{2\}}$. If there exists $Q^{[p]}:=\inf_{\omega\geq0}q_{\omega}^{[p]}$ satisfying $J_{\omega}^{[p]}(x)\stackrel{\text{unif}}{=}\mathcal{O}(|x|_{p}^{-q_{\omega}^{[p]}})$ with respect to distorted angles, then the following holds for $\beta>1-Q^{[p]}$.
  \begin{equation}\label{D-D_p}
    D_{\beta}^{[p]}(s:x)-\mathcal{D}_{\beta}^{[p]}(s:x)=s^{\beta+\frac{2}{p}}p^{\beta+1}
    \Gamma(\frac{1}{p})^{2}\sum_{n\in\mathbb{Z}^{2}\setminus\{0\}}\frac{J_{\beta+1}^{[p]}
    (2\pi\sqrt[p]{s}(x-n))}{(2\pi \sqrt[p]{s}|x-n|_{p})^{\beta+1}}
    \quad\text{for }s>0,\ x\in\mathbb{R}^{2}.\vspace{-5pt}
  \end{equation}}
  Furthermore, under this assumption, the series converges absolutely for $x\in\mathbb{T}^{2}$.
  \end{theorem}
  This result was derived using Poisson summation formula and Fourier inversion under the assumed integrability on $\mathbb{R}^{2}$ of a certain function corresponding to $\beta$. 
  We note that, by ($\ref{D2}$) with $\beta=0$, $x=0$ and $s=r^{p}$, the left-hand side of (\ref{D-D_p}) is $P_{p}(r)$, while, in the case $p=2$, the range of $\beta$ for which the theorem holds can be specifically identified as $\beta>\frac{1}{2}$ by asymptotic evaluations of the Bessel functions. Thus, the theorem cannot be applied to $\beta=0$, which means that Hardy's identity is not included in the claim of Theorem \ref{thm_At} at least for $p=2$. \par
  However, if we formally apply Theorem \ref{thm_At} to $\beta=0$ and transform the equation, we obtain the generalized Hardy's identity (\ref{HI-p}). \par
  Hence, the proof of the main result in this paper (Theorem \ref{GHI}) justifies the identity, but the convergence of (\ref{HI-p}) has not yet been confirmed except for the conditional convergence for $p=2$. From the assumptions of Theorem \ref{thm_At}, we can probably infer that (\ref{HI-p}) does not converge absolutely even for the general $p$, but a rigorous solution of this problem requires a uniformly asymptotic evaluation formula of $J_{\omega}^{[p]}$ on $\mathbb{R}^{2}$. \par
  We have already obtained the desired uniformly asymptotic evaluation for the case $\omega=0$ (\cite{K2}, Theorem 1.5)
  \begin{equation*}
      J_{0}^{[p]}(x)\stackrel{\text{unif}}{=}
      \begin{cases}
      \mathcal{O}(|x|_{p}^{-\frac{1}{2}}) & (p=2),\\
      \mathcal{O}(|x|_{p}^{-\frac{p}{2}}) & (\frac{2}{p}\in\mathbb{N}\setminus{\{1,2\}}),
      \end{cases}
      \qquad\text{as }|x|_{p}\to\infty,
  \end{equation*}  
  but for the cases $\omega>0$, it is considered to be difficult to derive the evaluation formulas by real analysis, especially by the method using oscillatory integrals. \par
  On the other hand, if we can derive the asymptotic expansion of $\mathcal{J}_{\omega,\varphi}^{[p]}$ following the complex-analytic method by H. Hankel \cite{Hankel} via analytic continuation of $\mathcal{J}_{\omega,\varphi}^{[p]}$ defined on non-negative real numbers in this paper, we expect to obtain uniformly asymptotic evaluation formulas with respect to $\varphi$ as $\mathcal{J}_{\omega,\varphi}^{[p]}(r)=\mathcal{O}(r^{-q_{\omega}^{[p]}})$. As a byproduct of this trial, in particular from the $\varphi$-invariant $\mathcal{J}_{1,\varphi}^{[p]}(r)=\mathcal{O}(r^{-q_{1}^{[p]}})$, the results of this paper allow us to conjecture
  \begin{equation}
    P_{p}(r)=\mathcal{O}(r^{1-q_{1}^{[p]}+\varepsilon})\qquad\text{for any small }\varepsilon>0,
  \end{equation}
  as the optimal $\mathcal{O}$ evaluation for $p$ satisfying $(2/p)\in\mathbb{N}$, considering the Hardy's conjecture $P_{2}(r)=\mathcal{O}(r^{\frac{1}{2}+\varepsilon})$. 
  
  %% ===========================================================================
  \section*{Acknowledgements}
  \hspace{13pt}The author is financially supported by JST SPRING, Grant Number JPMJSP2125, and would like to take this opportunity to thank the ``THERS Make New Standards Program for the Next Generation Researchers'' for excellent research conditions during the preparation of this paper. Also, I would like to express my gratitude to Prof. Mitsuru Sugimoto for numerous constructive suggestions and helpful remarks on harmonic analysis. 
  \section*{Data availability}
  \hspace{13pt} No data was used for the research described in the article. 
   
  \itshape{
  \hspace{13pt}The author's affiliation: Graduate School of Mathematics, Nagoya University, Chikusa-ku, Nagoya 464-8602, Japan\par
  The author's email address: kitajima.masaya.z5@s.mail.nagoya-u.ac.jp}

\addcontentsline{toc}{section}{References}
  \begin{thebibliography}{99}
    \bibitem{AomotoIguchi} K. Aomoto, K. Iguchi, On quasi-hypergeometric functions, Methods and Appli. of Analysis, \textbf{6} (1999) 55-66, \url{https://doi.org/10.4310/MAA.1999.v6.n1.a4}.
    \bibitem{Gauss} C.F. Gauss,  De nexu inter multitudinem classium, in quas formae
     binariae secundi gradus distribuuntur, earumque determinantem, In: Schering, E.,
      ed., Werke, Vol. 2. G\"{o}ttingen: K\"{o}niglichen Gesellschaft der Wissenschaften  (1876) 269-291, \url{https://doi.org/10.1017/CBO9781139058230.012}.     
    \bibitem{Hankel} H. Hankel, Die Cylinderfunctionen erster und zweiter Art, Math. Ann. \textbf{1} (1869) 467-501, \url{https://doi.org/10.1007/BF01445870}.
    \bibitem{Hardy-1915} G.H. Hardy, On the expression of a number as the sum of two squares, Q. J. Math. \textbf{46} (1915) 263-283, \url{http://refhub.elsevier.com/S0022-1236(21)00354-2/bibE0749EAC038C473558FC31E2A34CC3E5s1}.
    \bibitem{Hardy-1917} G.H. Hardy, E. Landau, The average order of the arithmetical
     functions $P(x)$ and $\Delta (x)$, Proc. Lond. Math. Soc. \textbf{15} (1917) 192-213, \url{https://doi.org/10.1112/plms/s2-15.1.192}.
    \bibitem{Huxley-2003} M.N. Huxley, Exponential sums and lattice points.
     $\mathrm{III}$, Proc. Lond. Math. Soc. (3) \textbf{87}(3) (2003) 591-609, \url{https://doi.org/10.1112/S0024611503014485}.
    \bibitem{Ivic-survey} A. Ivi\'{c}, E. Kr\"{a}tzel, M. K\"{u}hleitner, W.G. Nowak, Lattice points in large regions and related arithmetic functions: Recent developments in a very classic topic, arXiv:math/0410522v1 (2004), \url{https://doi.org/10.48550/arXiv.math/0410522}.
    \bibitem{FDE} A.A. Kilbas, H.M. Srivastava, J.J. Trujillo, Theory and Applications of Fractional Differential Equations, Elsevier. North-Holland mathematics studies. \textbf{204} (2006).
    \bibitem{K1} M. Kitajima, Series expansions by generalized Bessel functions for certain functions related to the lattice point problems for the $p$-circle, arXiv:2408.02613v2 (2024), \url{https://doi.org/10.48550/arXiv.2408.02613}.
    \bibitem{K2} M. Kitajima, Asymptotic evaluations of generalized Bessel function of order zero related to the $p$-circle lattice point problem, arXiv:2411.10850v4 (2024), \url{https://doi.org/10.48550/arXiv.2411.10850}.
    \bibitem{Kratzel} E. Kr\"{a}tzel, Lattice Points, Kluwer Academic Publication, 
    (1988).
    \bibitem{Kuba} G. Kuba, On sums of two $k$-th powers of numbers in residue classes $\mathrm{II}$, Abh. Math. Sem. Univ. Hamburg \textbf{63} (1993) 87-95, \url{https://doi.org/10.1007/BF02941334}.
    \bibitem{Kuratsubo-2022} S. Kuratsubo, E. Nakai, Multiple Fourier series and 
    lattice point problems, J. Func. Anal. \textbf{282} (2022) 1-62, \url{https://doi.org/10.1016/j.jfa.2021.109272}.
    \bibitem{Laugesen1} R.S. Laugesen, S. Ariturk, Optimal stretching for lattice points under convex curves, Port. Math. \textbf{74} (2017) 91-114, \url{https://doi.org/10.4171/PM/1994}.
    \bibitem{Laugesen2} R.S. Laugesen, S. Liu, Optimal stretching for lattice points and eigenvalues, Ark. Mat. \textbf{56} (2018) 111-145, \url{https://dx.doi.org/10.4310/ARKIV.2018.v56.n1.a8}.
    \bibitem{Stein-1971} E.M. Stein, G. Weiss, Introduction to Fourier Analysis on 
    Euclidean Spaces, Princeton Univ. Press, (1971).
    \bibitem{Watson} G.N. Watson, A treatise on the theory of Bessel functions, 
    2nd ed., Cambridge Univ. Press, (1995).
  \end{thebibliography}
\end{document}